\documentclass[]{moduli}

\renewcommand{\contentsname}{Contents\\{\footnotesize\normalfont(A table
of contents should normally not be included)}}

\usepackage{amsfonts}
\usepackage{amscd}
\usepackage{amsmath, mathrsfs, amssymb}
\usepackage{amsthm}
\usepackage{hyperref}
\usepackage{color}
\usepackage{graphicx}
\usepackage{graphicx,transparent}
\usepackage{enumerate}
\usepackage{caption}
\usepackage{bm}
\usepackage{xcolor}
\usepackage{cleveref}

\theoremstyle{plain}
\newtheorem{theorem}{Theorem}[section]

\newtheorem{conjecture}[theorem]{Conjecture}

\theoremstyle{definition}
\newtheorem{remark}[theorem]{Remark}

\newtheorem{example}[theorem]{Example}

\newcommand{\calC}{\mathcal C}

\newcommand{\calO}{\mathcal O}

\newcommand{\calP}{\mathcal P}

\newcommand{\calH}{\mathcal H}

\newcommand{\bbA}{\mathbb A}

\newcommand{\calL}{\mathcal L}

\newcommand{\bbP}{\mathbb P}

\newcommand{\Proj}{\operatorname{Proj}}

\begin{document}

\title{Gorenstein contractions of multiscale differentials}
\author{Dawei Chen}
\email{dawei.chen@bc.edu}
\orcid{0000-0002-3926-2636} 
\address{Department of Mathematics, Boston College, Chestnut Hill, MA 02467, USA}

\author{Qile Chen}
\email{qile.chen@bc.edu}
\orcid{0000-0001-6776-242X}
\address{Department of Mathematics, Boston College, Chestnut Hill, MA 02467, USA}

\thanks{Research of D.C. was supported in part by the National Science Foundation under Grant DMS-2301030, Simons Travel Support for Mathematicians, and a Simons Fellowship under Record ID SFI-MPS-SFM-00005694. Research of Q.C. was supported in part by the National Science Foundation under Grant DMS-2001089 and Simons Travel Support for Mathematicians.}

\begin{abstract}
Multiscale differentials arise as limits of holomorphic differentials with prescribed zero orders on nodal curves. In this paper, we address the conjecture concerning Gorenstein contractions of multiscale differentials, originally proposed by Ranganathan and Wise and further developed by Battistella and Bozlee. Specifically, in the case of a one-parameter degeneration, we show that multiscale differentials can be contracted to Gorenstein singularities, level by level, from the top down. At each level, these differentials descend to generators of the dualizing bundle at the resulting singularities. Moreover, the global residue condition, which governs the smoothability of multiscale differentials, appears as a special case of the residue condition for descent differentials. 
\end{abstract}

\maketitle

\section{Introduction}
\label{sec:intro}

Let $\mu = (m_1, \ldots, m_n)$ be a partition of $2g-2$, where $m_i \ge 0$ for all $i$ and $\sum_{i=1}^n m_i = 2g-2$. Let $\calH(\mu)$ denote the moduli space of holomorphic differentials $(C, \omega)$ of type $\mu$, where $C$ is a smooth, connected, and compact complex algebraic curve of genus $g$, and $\omega$ is a holomorphic differential on $C$ with $n$ distinct marked zeros whose orders are prescribed by $\mu$. The study of holomorphic differentials with specified zero orders plays a significant role in moduli theory and surface dynamics. We refer to \cite{Z06, W15, C17} for an introduction to this fascinating subject.  

On the one hand, the degeneration of holomorphic differentials to nodal curves is an important and delicate problem that has been understood fairly well thanks to a series of recent developments \cite{Ch17, G18, FP18, BCGGM18, BCGGM19, BCGGM, CC19, CGHMS22}. The answer to this problem has two components. Using the perspectives of limit linear series or logarithmic geometry, one obtains the so-called generalized multiscale differentials, which serve as virtual limits of the corresponding degenerations. Moreover, a generalized multiscale differential arises as an actual limit if and only if it additionally satisfies the so-called global residue condition.

On the other hand, far less is known about degenerations of differentials to curves with singularities more complicated than nodes. In this paper, we investigate an interesting interplay between these two aspects. 

We first review some related concepts and notations. Suppose $(X, \eta)$ is a generalized multiscale differential of type $\mu$, where $X$ is a pointed nodal curve and $\eta$ satisfies conditions (1), (2), and (3) in \cite[Definitions 1.1 and 1.2]{BCGGM18} (see also \cite{CGHMS22} for the logarithmic description of generalized multiscale differentials). In particular, $(X,\eta)$ admits a level graph $\Gamma$ with levels $0, -1, \ldots, -L$. Denote by $X_i$, $X_{>i}$, and $X_{<i}$ the subcurves of $X$ that lie at level $i$, above level $i$, and below level $i$, respectively. We also denote by $\eta_i$ the part of $\eta$ at level $i$, where $\eta_i$ has poles at the nodes joining $X_i$ to higher levels and is considered to be identically zero below level $i$. 

Next, we review how multiscale differentials arise as limits of ordinary differentials of type $\mu$. To make the exposition transparent and accessible to readers with diverse backgrounds in classical algebraic geometry and flat surface geometry, we focus on the case of a one-parameter degeneration. Suppose the structure of $(X, \eta)$ is obtained from a one-parameter family of differentials $(C_t, \omega_t)$ in $\calH(\mu)$ as $t$ approaches zero. Then, to each level $i$, we can assign an integer $\ell_i$, called the vanishing order at level $i$, such that 
$$0 = \ell_0 < \ell_{-1} < \cdots < \ell_{-L}$$ 
and such that $\eta_i$ is the limit of $t^{-\ell_i} \omega_t$ restricted to $X_{i}$. Moreover, if an edge $e$ in the level graph $\Gamma$ joins two vertices on levels $i$ and $j$, where $i > j$, then the slope $\kappa$ of $e$ (also called the number of prongs in \cite{BCGGM} or the contact number in \cite{CC19}) is given by 
$$\kappa = \frac{\ell_j - \ell_i}{a}, $$ 
where $a$ is the edge length of $e$. Geometrically, $a$ determines the singularity type $uv = t^a$ in the universal curve $\calC$ at the node corresponding to $e$. In this situation, $\eta$ has a zero of order $\kappa - 1$ on the upper branch of the node and a pole of order $-\kappa - 1$ on the lower branch. We remark that analogous structures hold for higher-dimensional families of (generalized) multiscale differentials as well. 

For a projective variety $S$, if the dualizing sheaf of $S$ is a line bundle, we say that $S$ is Gorenstein. If $S$ is a family of Gorenstein curves over a base, we denote by $\omega_S$ the (relative) dualizing bundle of the family. A conjectural relation between Gorenstein contractions of curves and smoothable generalized multiscale differentials (up to suitable semistable modification) was stated in \cite[Conjecture 5.1]{BB23} (see also \cite[Conjecture 1.2]{B24}) from the logarithmic and tropical viewpoints, where the conjecture was attributed to Ranganathan and Wise. This conjecture was verified for hyperelliptic differentials in \cite{BB23}. Further evidence, as well as a classification of Gorenstein curve singularities in genus three, was provided in \cite{B24}. 

Before addressing the conjecture in general, we explain the precise meaning of semistable modification, which consists of two parts. 

First, if $e$ is a long edge in the level graph $\Gamma$ that crosses more than one level, then we subdivide $e$ by adding a semistable rational vertex at each level crossed by $e$. We call such a vertex a semistable chain vertex. If the slope of $e$ is $\kappa$, then each semistable chain vertex carries a differential with a zero of order $\kappa-1$ on the upper edge and a pole of order $-\kappa-1$ on the lower edge. Geometrically, this modification blows up the node $e$ successively in the universal curve $\calC$ to a chain of semistable rational curves (see \cite[Section 3.5]{CGHMS22} for more details). In general, the lengths of the newborn short edges can be rational numbers, and we may apply a base change $t = s^d$, where $d$ is their common denominator, to make these edge lengths integers.
 
Next, for each marked zero $z$ belonging to a vertex $v$ in $\Gamma$, blowing it up successively in the universal curve expands it to a chain of semistable rational vertices starting from $v$, going downward in the level graph, and ending at a semistable leaf vertex that carries $z$. If the zero order of $z$ is $m$, then each semistable vertex in this sequence admits a differential with a zero of order $m$ and a pole of order $-m-2$, and the slope of every edge in the semistable sequence is $\kappa = m+1$. If the initial vertex $v$ is at level $h$, then for the $k$th semistable vertex in the sequence, we place it at the lower level with vanishing order $\ell = \ell_h + \kappa\sum_{j=1}^k a_j$, where $xy = t^{a_j}$ is the local singularity type of the $j$th edge in the semistable sequence obtained from the successive blowups. Furthermore, we can make the semistable sequence long enough so that every stable vertex above a chosen level $i$ in $\Gamma$ lies above a leaf vertex (which suffices for the purpose of contracting the subcurves above the chosen level $i$ as in the statement of Conjecture~\ref{conj} below). Finally, if any edge in the semistable sequence crosses more than one level, we subdivide it as described in the preceding paragraph.  

We denote by $(\tilde{X}, \tilde{\eta}, \tilde{\Gamma})$ the resulting modification of $(X, \eta, \Gamma)$. We also denote by $\tilde{\ell}_i$ the vanishing order at level $i$ in $\tilde{\Gamma}$. The upshot of the modification is that every edge in $\tilde{\Gamma}$ joins two adjacent levels, and every stable vertex (above a chosen level $i$) lies above all marked zeros. With these preparations, in terms of one-parameter smoothing, the aforementioned conjecture can be formulated as follows. 

\begin{conjecture}
\label{conj}
{\rm (i)} Suppose the semistable modification $(\tilde{X}, \tilde{\eta})$ of a generalized multiscale differential $(X, \eta)$ arises as the limit of a one-parameter family $\tilde{\calC}$ of holomorphic differentials $(C_t, \omega_t)$ in $\calH(\mu)$ as $t$ approaches zero. Then, for every level $i$ above all leaf vertices in $\tilde{\Gamma}$, there exists a reduced Gorenstein contraction $\sigma\colon \tilde{\calC} \to \calC'$ such that $\sigma$ contracts exactly $\tilde{X}_{>i}$ and the differential $\tilde{\eta}_i$ at level $i$ descends to a local generator of $\omega_{X'}$, where $X'$ is the central fiber of the one-parameter family $\calC'$ and $X'$ has isolated Gorenstein singularities formed by contracting each connected component of $\tilde{X}_{>i}$. Moreover, $\sigma^{*}\omega_{\calC'} = \omega_{\tilde{\calC}}\big(\sum_{j > i} (\tilde{\ell}_i - \tilde{\ell}_j) \tilde{X}_{j} \big)$. 

{\rm (ii)} Conversely, if for every level $i$, there exists a contraction $\varsigma\colon \tilde{X} \to X'$ such that $\varsigma$ contracts each connected component of $\tilde{X}_{>i}$ to an isolated singularity of $X'$ and $\tilde{\eta}_i$ descends to a local section of $\omega_{X'}$ at the resulting singularities, then $(X, \eta)$ can be smoothed into $\calH(\mu)$. 
\end{conjecture}

A proof of Conjecture~\ref{conj} will be given in Section~\ref{sec:proof}. In Remark~\ref{rem:higher}, we will also discuss possible approaches, as well as pending issues, to address this conjecture over a higher-dimensional base. 

We remark that both parts of the semistable modification in the assumption of the conjecture are necessary; otherwise, certain parts of the conjecture do not hold, which we will explain as follows. 

If there exists a long edge $e$ joining $X_{>i}$ and $X_{<i}$, then contracting $X_{>i}$ produces a singularity that touches $X_{<i}$, on which $\eta_{i}$ is identically zero; hence the descent of $\eta_i$ cannot be a local generator of the dualizing sheaf at the singularity (see Example~\ref{ex:chain} below). Nevertheless, this issue does not occur if we subdivide $e$ by adding semistable chain vertices and contract $\tilde{X}_{>i}$ instead. 

Additionally, if a marked zero section $Z$ of $(C_t, \omega_t)$ passes through $X_{>i}$ in the universal curve, contracting $X_{>i}$ makes the image of $Z$ pass through the resulting singularity, which can fail to be Gorenstein (see Example~\ref{ex:leaf} below). Nevertheless, this issue does not occur if we separate $Z$ from $X_{>i}$ by making it pass through a lower semistable leaf vertex in $\tilde{X}$ instead.

We further point out that each semistable vertex possesses a unique polar edge whose residue is zero. Therefore, the semistable modification does not affect the smoothability of $(X, \eta)$, which is governed by the global residue condition (GRC) (see \cite[Definition 1.2 (4)]{BCGGM18}). 

The ideas in this paper can lead to various applications to deformations of Gorenstein singularities, alternative modular compactifications, and the log minimal model program for moduli spaces of curves and differentials (in the sense of \cite{S13, AFS16, BC23, BKN23, CY25, NW25}). We plan to address these questions in future work. 

\section{Proof of the Conjecture}
\label{sec:proof}

\begin{proof}[Proof of Conjecture~\ref{conj}~{\rm (i)}]
By \cite[Proposition 2.6]{S13}, there exists a (unique) birational contraction $\sigma \colon \tilde{\calC} \to \calC'$ such that $\calC'$ is normal and the subcurve $\tilde{X}_{>i}$ in the central fiber $\tilde{X}$ is the exceptional locus of $\sigma$. In particular, each semistable chain vertex above level $i$ is contracted, whereas all semistable leaf vertices remain uncontracted. 

Denote by $Y$ a connected component of $\tilde{X}_{>i}$ that is contracted to a singularity $p$ in the central fiber $X'$ of $\calC'$. For every regular function $f \in \calO_{X', p}$, since $\calO_{\calC'} \twoheadrightarrow \calO_{X'}$, there exists a regular function $h \in \calO_{\calC'}(U)$ for a small neighborhood $U$ of $p$ such that $h|_{X' \cap U} = f$. Note that $\sigma^{-1}(U)$ is an open subset containing $Y$ in $\tilde{\calC}$. For $t$ near zero, let $Y_t \subset \tilde{\calC}_t$ be a family of subsurfaces converging to $Y$ in $U$, and denote by $V_t$ the union of the vanishing cycles in the boundary of $Y_t$ that shrink to the nodes joining $Y$ and $\tilde{X}\setminus Y$. Here, the subsurface $Y_t$ as $t$ approaches zero can be defined by deleting a neighborhood of the core of the flat surface representing the differential $t^{-\tilde{\ell}_i}\omega_t$ at level $i$, where the neighborhood shrinks with $t$ under the flat metric of $t^{-\tilde{\ell}_{i+1}}\omega_t$ above level $i$, and the boundary of the neighborhood consists of the vanishing cycles $V_t$ (see \cite[Figures 14 and 15]{BCGGM18} for an example). 

By the definition of generalized multiscale differentials, we have 
$$\lim_{t\to 0} (t^{-\tilde{\ell}_i} \omega_{t})|_{\tilde{X}_{i}} = \tilde{\eta}_i, $$
where $\tilde{\ell}_i$ is the vanishing order of the family of differentials $(\omega_t)$ on $\tilde{X}_{i}$.  
Let $h_t = h|_{U\cap \calC'_t} = (\sigma^{*}h)|_{U\cap \calC_t}$ for $t$ near zero. Since removing $V_t$ disconnects $Y_t$ from $\calC_t\setminus Y_t$, the homology class of $V_t$ is trivial, which can be verified by using the standard polygonal representation of the topological surface $Y_t$ with boundary $V_t$.  
Therefore, for $t$ near zero, the holomorphic differential $h_t \cdot t^{-\tilde{\ell}_i} \omega_{t}$ on $Y_t$ satisfies 
\begin{eqnarray}
\label{eq:integral}
\int_{V_t} h_t \cdot t^{-\tilde{\ell}_i} \omega_{t} = 0.
\end{eqnarray}
Viewing $Y_t$ as a Riemann surface with boundary $V_t$, the above vanishing also follows from Stokes' Theorem since $d\omega_t = 0$. 

Observe that as $t$ approaches zero, the limit $V$ of $V_t$ consists of nodes $p_i$, where each $p_i$ maps to $p$ under $\sigma$. Furthermore, the differential $h_t \cdot t^{-\tilde{\ell}_i} \omega_{t} $ converges to $\varsigma^{*}f \cdot \tilde{\eta}_i$ on $\tilde{X}_{i}$, where $\varsigma = \sigma|_{\tilde{X}}\colon \tilde{X} \to X'$. 
Therefore, using the same argument as in the proof of \cite[Theorem 1.3]{BCGGM18}, \eqref{eq:integral} implies that  
\begin{eqnarray}
\label{eq:descend}
\sum_{p_i\in V}  {\rm Res}_{p_i} (\varsigma^{*}f \cdot \tilde{\eta}_i) = 0, 
\end{eqnarray}
which means that $\tilde{\eta}_i$ descends to a regular section of the dualizing sheaf $\omega_{\calC'}$ (see \cite[\S1.3.1]{B24} for more details and related references). 
 
Since $Y$ does not intersect any marked zero section of $(\omega_t)$, which has been taken away by a lower semistable leaf vertex, the differential form $t^{-\tilde{\ell}_i} \omega_t \wedge dt$ is holomorphic and nowhere vanishing in $U\setminus p$. Since the codimension of $p$ in $U$ is two, this implies that there exists a Cartier canonical divisor in $U$ which does not contain $p$. Therefore, $\omega_{\calC'}$ is locally a line bundle at $p$ (see \cite[Proposition 5.75]{KM98} and \cite[Section 5]{K13}). Hence, $\calC'$ is Gorenstein at $p$, and consequently the central fiber $X'$ is also Gorenstein at $p$. 

Since $\tilde{\eta}_i$ is the limit of $t^{-\tilde{\ell}_i} \omega_t$ on $\tilde{X}_{i}$, the descent of $\tilde{\eta}_i$ is a regular local section of $\omega_{X'}$ which does not vanish in $X' \cap U$. It follows that $\tilde{\eta}_i$ locally generates $\omega_{X'}$ at the singularities formed by contracting each connected component of $\tilde{X}_{>i}$. 

Finally, the twisted dualizing bundle $\omega_{\tilde{\calC}}\big(\sum_{j > i} (\tilde{\ell}_i - \tilde{\ell}_j) \tilde{X}_{j} \big)$ restricted to $\tilde{X}_j$ for $j > i$ is isomorphic to 
$\omega_{\tilde{C}}|_{\tilde{X}_j}\big( \sum_k \kappa_{j+1, k} q_{j+1, k}-\sum_k \kappa_{j-1, k} q_{j-1, k}\big)$, which is trivial, where $q_{j-1, k}$ and $q_{j+1, k}$ range over the nodes joining $\tilde{X}_j$ to the adjacent upper level curve $\tilde{X}_{j+1}$ and to the adjacent lower level curve $\tilde{X}_{j-1}$, respectively. Similarly, $\omega_{\tilde{\calC}}\big(\sum_{j > i} (\tilde{\ell}_i - \tilde{\ell}_j) \tilde{X}_{j} \big)$ restricted to $\tilde{X}_i$ is isomorphic to $\omega_{\tilde{C}}|_{\tilde{X}_i}(\sum_k \kappa_{i+1, k} q_{i+1, k})$, which admits $\tilde{\eta}_i$ as a regular section, and $\omega_{\tilde{\calC}}\big(\sum_{j > i} (\tilde{\ell}_i - \tilde{\ell}_j) \tilde{X}_{j} \big)$ restricted to $\tilde{X}_{< i}$ 
is isomorphic to $\omega_{\tilde{\calC}}|_{\tilde{X}_{< i}}$.  Since $\sigma^{*}\omega_{\calC'}$ satisfies the same properties and differs from $\omega_{\tilde{\calC}}$ by twisting along the exceptional locus of $\sigma$ supported on $\tilde{X}_{>i}$, it follows that 
$\sigma^{*}\omega_{\calC'} = \omega_{\tilde{\calC}}\big(\sum_{j > i} (\tilde{\ell}_i - \tilde{\ell}_j) \tilde{X}_{j} \big)$ (see \cite[Lemma 2]{S82} for the uniqueness of such an expression). 
\end{proof}

\begin{proof}[Proof of Conjecture~\ref{conj}~{\rm (ii)}]
Since $\tilde{\eta}_i$ descends, it satisfies  
$$\sum_{p_i\in \varsigma^{-1}(p)}  {\rm Res}_{p_i} (f \cdot \tilde{\eta}_i) = 0$$ 
for every regular function $f\in \calO_{X',p}$, where $p$ is the singularity in $X'$ formed by contracting a connected component of $\tilde{X}_{>i}$. Taking $f$ to be the constant function $1$, this recovers the global residue condition (GRC) in \cite[Definition 1.2 (4)]{BCGGM18}. Therefore, $(X, \eta)$ can be smoothed into $\calH(\mu)$, since the GRC does not change under the semistable modification.   
\end{proof}

\section{Examples and Remarks}
\label{sec:ex-rem}

In this section, we present several examples and remarks, to help the reader better understand the various conditions in Conjecture~\ref{conj}, as well as possible generalizations over higher-dimensional bases.  

\begin{example}[(Semistable chain vertices are necessary)]
\label{ex:chain}
Suppose $(X, \eta)$ is given by the triangular graph on the left of \cite[Figure 3]{CGHMS22}, which has a long edge with the slope labeled by $\kappa_3$. In the case of level $i = -1$ in the conjecture, if we contract $X_0$ to a singularity $p$, then the image of $X_{-2}$ passes through $p$. Since $\eta_{-1}$ is zero on $X_{-2}$, it cannot be a local generator of the dualizing sheaf at $p$. 

As explained previously, this issue can be resolved by subdividing the long edge, as in the graph on the right of \cite[Figure 3]{CGHMS22}.    
\end{example}

\begin{example}[(Semistable leaf vertices are necessary)]
\label{ex:leaf}
Consider $X = X_{0} \cup X_{-1}$, where $X_0$ has genus two, $X_{-1}$ has genus one, $(\eta_0) = z_1 + q$, $(\eta_{-1}) = 3z_2 - 3q$, $z_1\in X_0$ and $z_2\in X_{-1}$ are smooth points, and $q$ is the node. Then $(X, \eta)$ is a multiscale differential of type $\mu = (1,3)$ with $\kappa = 2$ at the node $q$. If $(X, \eta)$ can be smoothed via a one-parameter family $\calC$ of $(\calC_t, \omega_t)$ into $\calH(1,3)$ such that there is a birational contraction $\sigma\colon \calC \to \calC'$ with $X_0$ as the exceptional locus and with $\calC'$ Gorenstein, then the line bundle $\sigma^{*}\omega_{\calC'}$ is trivial on $X_0$, which must be given by $\omega_{\calC} (3X_{0})$ for degree reasons. Nevertheless, $\omega_{\calC} (3X_{0})$ restricted to $X_0$ has divisor class $z_1 + 2q - 3q = z_1 - q$, which is nontrivial for $z_1 \neq q$ and leads to a contradiction. 

As explained previously, this issue can be resolved by blowing up at $z_1$ to insert a rational leaf vertex carrying $z_1$ so that the contracted component $X_0$ does not contain any marked zero. Indeed, after the semistable modification, the bundle $\omega_{\tilde{\calC}}(2\tilde{X}_0)$ restricted to $\tilde{X}_0$ becomes trivial. 
\end{example}

\begin{example}[(Gorenstein singularities with $\mathbb G_m$-action and test configurations)]
\label{ex:test}
We consider a trivial one-parameter family of holomorphic differentials $(X, \eta)$ in a stratum $\calH(m_1, \ldots, m_n)$, where the underlying family of curves is $X\times \bbA^1$. Let $\kappa_i = m_i +1$ and choose positive integers $a_i$ such that $a_1 \kappa_1 = \cdots = a_n \kappa_n$. Perform a weighted blowup at each zero in the central fiber such that the resulting semistable rational tail is attached at a singularity of type $x_iy_i = t^{a_i}$ for $i = 1, \ldots, n$. The blown-up ideal can be described explicitly as $(x_i, t^{a_i})$, where $x_i$ is the fiber coordinate at each zero 
and $t$ is the base coordinate. In this case, $\ell_i = a_i \kappa_i$ have the same value $\ell$ for all $i$. Therefore, we can contract the proper transform $\tilde{X}$ and obtain a Gorenstein singularity with $n$ rational branches, which admits a natural $\mathbb G_m$-action. 

Indeed, this construction yields a test configuration for the (logarithmic) dualizing bundle of $X$, where the associated Rees algebra can be used to determine the weights and characters under the $\mathbb G_m$-action (see \cite{BHJ17} and \cite{CY25}).
\end{example}

\begin{remark}[(Partial orders below contracted levels)]
\label{rem:semistable}
Given a level $i$, if our goal is to contract components above level $i$ and obtain Gorenstein singularities as well as the descent of $\tilde{\eta}_{i}$ (rather than achieving this for every level), then we only need to assume that the GRC holds above level $i$, namely, we only need that the multiscale differential above level $i$ can be smoothed out. In particular, we do not need to fully order the vertices below level $i$. Such partially ordered curves correspond to the centrally aligned curves in \cite[Section~4.6.2]{RSPW19}.

Similarly, in this case, we only need to break up long edges that cross level $i$ by adding a semistable chain vertex at level $i$, and push marked zeros that are above level $i$ via a semistable chain down to a leaf vertex at level $i$. Then the same argument works as in the proof of Conjecture~\ref{conj} without any change.
\end{remark}

\begin{remark}[(Contraction from the bottom up)]
As we have seen, contracting multiscale differentials from the top down can ensure that the resulting objects have a nice algebraic structure. 

In contrast, suppose we want to contract from the bottom up, e.g., by forgetting the components in lower levels. In terms of flat geometry, this means discarding the part of a surface that shrinks and disappears under the flat metric induced by the holomorphic differential. The corresponding moduli space is called the WYSIWYG compactification (What You See Is What You Get), which is only a topological compactification. In general, it does not admit any structure of an algebraic variety or a complex analytic space (see \cite{CW21}). 
\end{remark}

\begin{remark}[(Contraction of meromorphic differentials)]
Consider a one-parameter degeneration of meromorphic differentials to a multiscale differential. If there is a marked pole above level $i$ in the limit multiscale differential, then it cannot be pushed lower via a rational semistable chain. If we contract a connected component above level $i$ containing this marked pole, then the corresponding section of this pole passes through the resulting singularity, which prevents us from finding a nowhere vanishing regular differential in a neighborhood away from the singularity. 

In contrast, if all marked poles appear on or below level $i$, then they do not affect the desired contraction above level $i$. 
\end{remark}

\begin{remark}[(Logarithmic R-maps and the twisted dualizing bundle)]
\label{rem:infinity}
In the proof of Conjecture~\ref{conj}, the twisted dualizing bundle that induces the desired contraction can also be interpreted from the viewpoint of logarithmic R-maps (here ``R'' refers to R charge in physics, which corresponds to the dualizing bundle in our context).

Consider as before a one-parameter family of differentials $(\calC_t, \omega_t)$ degenerating to a multiscale differential $(\tilde{X}, \tilde{\eta})$, where $\tilde{\eta}_i$ is the limit of $t^{-\tilde{\ell_i}} \omega_t$ restricted to the component $\tilde{X}_i$ on each level $i$. Let $\bbP = \bbP(\omega_{\tilde{\calC}}\oplus \calO)$ be the natural compactification of the relative dualizing bundle over the universal curve, where $\bm{\infty}\subset \bbP$ denotes the infinity section, and let $\calP$ be its logarithmic stack relative to $\bm{\infty}$ (see \cite[Section 1.3]{CJRS22}). 

Given a level $i$, consider the logarithmic R-map $f\colon \tilde{\calC}\to \calP$ 
which, near $\tilde{X}_j$ for each $0\geq j > i$, is 
induced by $f (t, x) = \big(t, x, [\omega_t(x), t^{\ell_i - \ell_j}]\big)$, and near $\tilde{X}_j$ for each $j \leq i$ is induced by $f (t, x) = (t, x, [\omega_t(x), 1])$. In other words, the components on level $j > i$ fall into $\bm{\infty}$ with multiplicity $\ell_i - \ell_j$, the component on level $i$ does not fall into $\bm{\infty}$ entirely and only meets $\bm{\infty}$ at the marked zeros and nodes on level $i$, and the components below level $i$ fall into the zero section of $\calP$. 

In this context, the twisted dualizing bundle $\omega_{\tilde{\calC}}\big(\sum_{j > i} (\tilde{\ell}_i - \tilde{\ell}_j) \tilde{X}_{j} \big)$ can be simply identified with $\omega_{\tilde{\calC}} \otimes f^{*} \calO (\bm{\infty})$. 
\end{remark}

\begin{remark}[(The conjecture over higher-dimensional bases)]
\label{rem:higher}

    Conjecture~\ref{conj} can be formulated similarly over a higher-dimensional base (see \cite[Conjecture 5.1]{BB23}). Here we briefly discuss possible approaches and pending issues toward resolving the conjecture in general.

First, using the ordering of the vertical boundary divisors in the moduli space of multiscale differentials (see \cite[Section 5]{CMZ22}), one can still define a twisted dualizing bundle $\tilde{\calL}$ on the universal family $\pi\colon \tilde{\calC}\to \tilde{B}$, where $\tilde{B}$ is the higher-dimensional base, such that $\tilde{\calL}$ is trivial when restricted to the fiber components to be contracted and positive when restricted to the other fiber components. Then we may expect that $\Proj \big(\bigoplus_{k \geq 0} \pi_{*}\tilde{\calL}^k\big)$ will produce the desired contraction.

To make this Proj construction work, it is required that $\pi_{*}\tilde{\calL}^k$ be locally free (at least for sufficiently large $k$). This is not a problem over a one-dimensional base, because over a DVR, any torsion-free sheaf is locally free. However, if the base $\tilde{B}$ has higher dimension, one must show that along every one-parameter degeneration to the same limit multiscale differential parametrized in $\tilde{B}$, the resulting contracted central fiber is independent of the degeneration direction.

In the case of genus one, this issue was resolved in \cite{RSPW19}, where a pointed genus-one curve can be regarded as carrying a canonical divisor of type $\mu = (0,\ldots, 0)$, i.e., having a “zero’’ of order $0$ at each marked point. Indeed, \cite[Section 3.7]{RSPW19} showed that when the family is a smoothing, $\pi_{*}\tilde{\calL}^k$ commutes with base change, and hence the problem can be reduced to the case of a one-parameter smoothing. However, in higher genus, several parts of the argument in \cite{RSPW19} require additional work. For example, the slope of an edge in the level graph can be greater than one in higher genus, which causes nonreduced structures supported on the fiber components to be contracted to appear in the relevant analysis. In particular, the inductive procedure of filtering into flat pieces, as in the proof of \cite[Lemma 3.7.4.5]{RSPW19}, needs to be revised in higher genus. 

Another case where the conjecture is known over an arbitrary-dimensional base is for hyperelliptic differentials, as shown in \cite{BB23}, where the argument relies on the hyperelliptic double cover of $\bbP^1$ and therefore reduces essentially to the degeneration of $\bbP^1$. However, this covering structure does not extend to general differentials.

A possible alternative approach is to first construct the contraction abstractly, and then verify its geometric properties. This is the method we employed in the one-parameter case, following \cite{S13}, which uses Artin’s criterion for contracting a (possibly reducible) curve with negative definite self-intersection matrix inside an algebraic surface (see \cite[Corollary (6.12) (b)]{A70} and also \cite{G62} in the complex-analytic setting). To apply this method more generally, we need an analog of Artin’s criterion for a smoothing family of nodal curves over a higher-dimensional base, ensuring not only that the contraction exists for the chosen fiber components, but also that the resulting family remains flat with reduced fibers after contraction. In other words, we need to ensure that such a global contraction, if it exists, produces reduced fiber curves that are consistent with those arising from every one-parameter degeneration in the base, rather than yielding more degenerate fibers with nonreduced structures or embedded points. 

In order not to obscure the main point of the paper, we plan to carry out the above ideas in more detail elsewhere. For now, we hope that the above discussion will motivate further developments in this circle of ideas, both toward resolving the conjecture in full generality and toward exploring the many related applications. 
\end{remark}

\begin{acknowledgements}
We are grateful to Jonathan Wise for inspiring discussions toward resolving the conjecture over higher-dimensional bases. We also thank Luca Battistella, Sebastian Bozlee, Maksym Fedorchuk, Samuel Grushevsky, David Holmes, Yuchen Liu, Martin M\"oller, Adrian Neff, Dhruv Ranganathan, Fei Yu, and Junyan Zhao for helpful communications on related topics. Finally, we thank the referees for their pertinent comments on an earlier version of this work. 
\end{acknowledgements}

\bibliographystyle{alpha}
\bibliography{biblio}
   
\end{document}